\documentclass[12pt,a4paper]{article}
\usepackage{cmap} 
\usepackage[utf8]{inputenc}
\usepackage{amsmath,amsfonts,amssymb,amsthm}
\usepackage[ukrainian,english]{babel}
\usepackage{wrapfig}
\usepackage{graphicx}
 \usepackage[colorlinks=true]{hyperref}

\theoremstyle{definition}

\newtheorem{theorem}{Theorem}

\newtheorem{lemma}{Lemma}

\newtheorem{remark}{Remark}
\newtheorem{corollary}{Corollary}

\begin{document}

\marginpar{\tiny 20260330}
%

\begin{center}
{\Large On the  non-uniqueness of continuous solutions to differential equations with a discrete state-dependent delay}

\medskip

Alexander Rezounenko\footnote{ORCID  \href{https://orcid.org/0000-0001-8104-1418}{https://orcid.org/0000-0001-8104-1418}}

Kharkiv National University, 4 Svobody sqr., Kharkiv, 61022, Ukraine

rezounenko@gmail.com
\end{center}

We discuss the non-uniqueness of continuous solutions to differential equations with a {\it discrete } state-dependent delay and continuous initial functions. We are interested not only in the fact (conditions) of non-uniqueness, but in additional information on the number of non-unique solutions and discuss an approach to classify them.  We provide a few explicit (easy to verify) examples  of  the non-uniqueness of continuous solutions and propose an approach to their classification. This partial classification may be illustrated by using just three collors and their simple and intuitive combination.  
We recognize that this initial classification is not exhaustive, and further study is necessary to build a complete picture.

\smallskip

Keywords: non-uniqueness; uniqueness; state-dependent delay;

\smallskip

Mathematics Subject Classification 2020: 
34K43; 34A12; 34K05.


\section{Introduction}

The study of state-dependent delay equations has a long history 
and constantly attracts attention of many researchers due to big number of applied problems  \cite{Hartung-Krisztin-Walther-Wu-2006}.

We are interested in solutions of initial valiue problems (Cauchy problems) for a delay equation with a {\it discrete } delay argument. Let us consider the problem in details. 

Let $D$ be a domain in $\mathbb{R}^n$. 
Consider the following {\it state-dependent delay} (SDD) differential equation (system)
\begin{equation}\label{sdd-winst-E}
\dot x(t) = F(x(t), x(t -g(x(t)))). 
\end{equation}
with  SDD functional 
  $g : D \to [0,h]$ 
 and $F: D \times D \to\mathbb{R}^n$, both continuous. 
Here $ \dot x(t)$ denotes the right-hand derivative of $x(t).$

Along with (\ref{sdd-winst-E}), we consider a simplified equation 
\begin{equation}\label{AR-sdd-2026_eq}
\dot{x}(t) = f(x(t - g(x(t))), 
\end{equation}
with a continuous $f: \mathbb{R} \to\mathbb{R}$ (or $f: \mathbb{R}^n \to\mathbb{R}^n$ in case of systems).

Basic properties of delay equations one may find in classical monographs \cite{Walther_book,Hale-book-1977}. For more details on SDD differential equations we refer to review \cite{Hartung-Krisztin-Walther-Wu-2006}. 
In the last review \cite{Hartung-Krisztin-Walther-Wu-2006} one may find a number of real-world applications. 

We are mainly interested in properties of uniqueness / non-uniqueness of solutions to (\ref{sdd-winst-E}) and (\ref{AR-sdd-2026_eq}). 
As will be discussed below, we will distinguish an {\it immediate non-uniqueness} and {\it waiting-time non-uniqueness}. Since our focus is on the immediate non-uniqueness (for the clarity of the essential properties), we are mainly consider small $t\ge 0$. 

For simplicity, we choose in (\ref{sdd-winst-E}) and (\ref{AR-sdd-2026_eq})
an autonomous form of delay equation (system), so we agreed that initial data are defined for $\theta \in [-h,0]$ (bounded delay interval) and   
an initial condition is
\begin{equation}\label{AR-sdd-2026_ic}
x(\theta) = \varphi(\theta), \quad \theta \in [-h,0]
\end{equation}
with a {\it continuous} initial function $\varphi$. We denote it by $\varphi\in C[-h,0]$. The multidimensional case  $\varphi\in C([-h,0]; \mathbb{R}^n)$ and the corresponding changes in (\ref{AR-sdd-2026_eq}) could be considered in a similar way. The case of unbounded delay ($h=+\infty$
)  can be considered in the same way. 
More precisely, the questions under consideration can be easily reformulated for {\it locally bounded delay} systems.

In literature, one may found important properties of solutions with state-dependent delays of more general form $r(x_t)$, where, as usual for delay equations,  $x_t\equiv x(t+\theta), \theta\in [-h,0]$ and $r: C([-h,0]; \mathbb{R}^n) \to [0,h]$. 
For more details  and references we refer to review \cite{Hartung-Krisztin-Walther-Wu-2006}.

The existence of solution to 
(\ref{sdd-winst-E}),  (\ref{AR-sdd-2026_ic}) 
is easily proven by the classical Brouwer's fixed-point theorem, provided $F, g, \varphi$ are continuous. The uniqueness needs more care.

It is not difficult to check that (local) Lipschitz conditions for  $
F, f, g$ and $\varphi$ implies the uniqueness of solutions (see e.g. \cite{Mallet-Paret} and references therein). 

If one is interested in dynamical properties of solutions, even more restrictions on the initial $\varphi$ may be useful  ($C^1$-solutions, solution manifold \cite{Walther-JDE-2003}).

An additional condition on the state-dependent delay introduced in \cite{Rezounenko_NA-2009} (see also \cite{Rezounenko-JMAA-2012})
allows to prove the uniqueness and well-posedness of solutions in the whole space of merely continuous initial functions $\varphi \in C[-h,0]$. The result indicates that the simplest 
form of SDD which may lead to the non-uniqueness is $r(x_t)=g(x(t))$. 

Taking into account the mentioned results, for arbitrary  {\it continuous} initial function $\varphi\in C[-h,0]$,  without additional condition on the SDD \cite{Rezounenko_NA-2009}, one cannot, in general,  guaranty the uniqueness of solution to 
(\ref{sdd-winst-E}),  (\ref{AR-sdd-2026_ic}) 
and even to (\ref{AR-sdd-2026_eq}), (\ref{AR-sdd-2026_ic}). 
It is clearly illustrated by the example 1 presented below  \cite{Driver-AP-1963}.

Let us explain our main ideas. 

In addition to the standard view of the non-uniqueness of solutions, when it is sufficient to prove the existence of more than one solution, we are interested in the following questions:

\smallskip

{\bf (A)} How many solutions does the Cauchy problem have?

\smallskip

{\bf (B)} In the case of many solutions, are all the solutions "the same", and if not, can they be classified in some way?

\smallskip

We offer for discussion partial answers to both questions (A) and (B).

A clear example of the Cauchy problem with an infinite (uncountable) number of solutions demonstrates the enormous variety of possible answers to question (A), even in the simplest, scalar case (see below).

Discussing question (B), from our point of view, non-unique solutions are not "the same". It seems quite natural to pay attention to the delay argument of each solution and classify the solutions according to the behavior of this argument.
For a more intuitive (visual) understanding, we propose to "paint" the solutions in different colors (blue, red, yellow), depending on the behaviour of the delay  argument.

Since any continuous solution to (\ref{sdd-winst-E}),  (\ref{AR-sdd-2026_ic}) is continuously differentiable for $t>0$, the (local) $C^1$ property of SDD $g$ implies the $C^1$ property of the delayed argument $s(t)\equiv t-g(x(t))$. A classical connection between the sign of $\dot s(t)$ and the monotonicity of $s(t)$, advises to look closer on solutions with a monotone delay argument. Another theorem from Calculus says that  in case  $\dot s_{+}(0)\neq 0$ the sign of $\dot s(t)$ is preserved (for small $t>0$), so one has a strictly monotone $s(t)$.  

As an additional, not the main, idea for discussion, we propose to consider the "graphs" of solutions in an extended coordinate system, in which, in addition to the independent argument $t$ (time) and the value of $x(t)$, we trace the (dependent) coordinate - the delay  argument $s(t)$.
As we will see below, such a view of the "extended graph" provides a surprisingly simple perception of the solutions. At the same time, the projections onto the classical subspaces ([time, solution value] and [time, delay  argument]) reproduces different sides of the classical picture.

We hope our colored / extended-graph point of view will be useful for a future study.

\bigskip

\section{Non-uniqueness of solutions}

Let us discuss a few particular, but important examples 1-3, which propose a particular point of view on the question of non-uniqueness of solutions.

\subsection{Example 1/1963.}

Let us give an example of {\it non-uniqueness} 
of solutions to the initial value problem for an ordinary differential equation with state-dependent delay. This example is given in the work \cite{Driver-AP-1963} of R. Driver 
(1963).

A simple scalar equation
\begin{equation}\label{Driver-example}
\dot y(t) = -2 y(t-y(t)) + 5,\quad  t>0
\end{equation}
with continuous initial function
\begin{equation}\label{Driver-example-ic}
y(t)= \varphi (t)= |4+t|^{1/2} + 2,\quad t\le 0
\end{equation}

has {\bf two} continuous (even infinitely differentiable) 
solutions  
\begin{equation}\label{Driver-example-solutions}
y^r(t)=4+t,\, t\ge 0\,  \mbox{ and } \,  y^y(t)=4+t-t^2,\, t\in [0,2].
\end{equation}
  \begin{wrapfigure}{r}{0.5\linewidth} 
    \centering
    \includegraphics[width=0.90\linewidth]{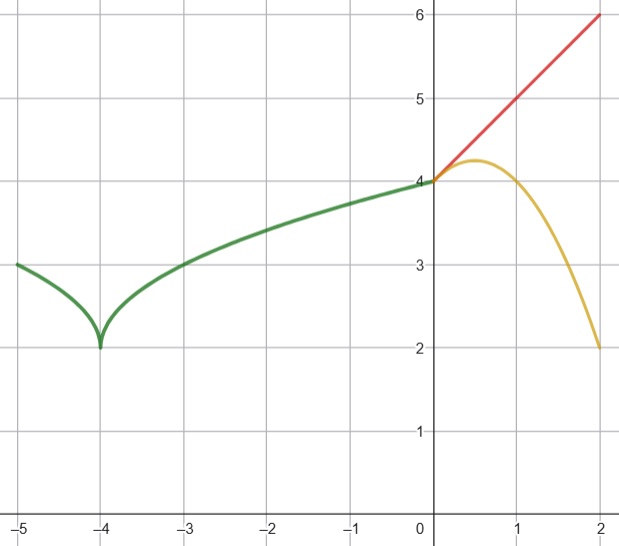}
    \caption{Example 1/1963}
    \label{fig:Driver-graph_1}
\end{wrapfigure}


The meaning of supscripts will be explained below. 



\medskip

The same idea (an example) was used in \cite[p.395]{Winston-JDE-1970} (1970). 

The analysis of this elegant Drivers`s example leads to different conclusions why the non-uniqueness appears. Naturally, the initial function is not Lipschitz and the SDD 
does not satisfy the additional condition \cite{Rezounenko_NA-2009}.

\subsection{Example 2/2010.} 

%

 Careful analysis of the Driver's example 1 allows to construct the following  one \cite{Rezounenko-metod-2010}. 
 
 Consider  equation 
\begin{equation}\label{metod2010-nu4}
\dot x(t) = \frac{1}{\alpha-1} \cdot x(t-|x(t)|) +1,\quad  t>0 
\end{equation}
%
with the initial function
(here $\alpha\in (0,1)$)
\begin{equation}\label{metod2010-nu7}
\hskip-30mm
x(t)=\varphi (t)\equiv \left\{ \begin{array}{ll}
-(-1-t)^\alpha, & \hbox{ if }\quad -2\le t\le -1;\\
(t+1)^{\alpha}, & \hbox{ if }\quad -1<t\le 0,
\end{array}
\right.
\end{equation}
The Cauchy problem (\ref{metod2010-nu4}), (\ref{metod2010-nu7})
 has 
{\it three} solutions
(for small $t>0$):
\begin{equation}\label{metod2010-nu8}
x^r(t)= t+1, \quad 
x^y(t)= t+1 -t^{\frac{1}{1-\alpha}}
\quad \hbox{ and }\quad x^b(t)= t+1
+t^{\frac{1}{1-\alpha}}.
\end{equation}

%
%

  \begin{wrapfigure}{r}{0.5\linewidth} 
    \centering
    \includegraphics[width=0.80\linewidth]{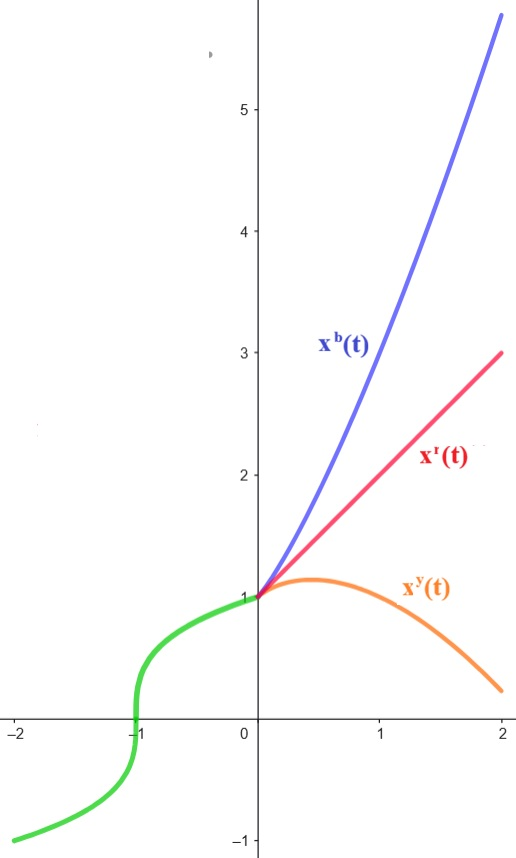}
    \caption{Example 2/2010}
    \label{fig:2}
\end{wrapfigure}


\begin{remark} 
One can see the similarities of Cauchy problems (\ref{Driver-example}), (\ref{Driver-example-ic}) and  (\ref{metod2010-nu4}),  (\ref{metod2010-nu7}), (take, for example, $\alpha = \frac{1}{2}$).
Nevertheless, there are three  solutions  (\ref{metod2010-nu8}) and only two solutions  (\ref{Driver-example-solutions}). 
Why there is no the third solution in the Driver's example?
\end{remark}

Since questions (A) and (B) are very general, we restrict ourself to the simplest case of {\it one} discrete state-dependent delay and continuous initial functions with only {\it one} point of non-Lipschitz. 
It is natural that several SDDs and several points of non-Lipschitz may multiply the number of solutions.  

To answer question (A) we construct the following example, which is a natural generalization of (\ref{metod2010-nu4}),  (\ref{metod2010-nu7}). 

\subsection{Example 3/ Key example of the non-uniqueness}

Consider the state-dependent delay differential equation
\begin{equation}\label{eq:main3}
\dot x(t) = -\,x\bigl(t - g(x(t))\bigr), \qquad t\ge 0,
\end{equation}
with a fixed continuous initial function $\varphi:[-h,0]\to\mathbb{R}$ 
such that $\varphi(0)=1$. 

Denote the delayed argument
\begin{equation}\label{eq:main3-s}
s(t) \equiv t - g(x(t)).
\end{equation}
At \(t=0\) we set \(s^0:=s(0)\).
We set $x(0)=\varphi(0)=1$. 
%
Consider $g(x)=|x|$, so $g(1)=1$ and $s(0)=-1$.

To have multiple solutions we need $s^0=s(0)=-1$ to be 
 {\it the point of non-Lipschitz} of initial function $\varphi$. This is obvioulsly, a necessary assumpion. Otherwise, the Lipshitz functions arguments leads to the uniqueness of continuous solutions.

Now we construct \(\varphi\) which has two different one-sided Hölder
behaviors at the single non-Lipschitz point \(s^0\):

\begin{equation}\label{eq:main3-ic}
\varphi(\theta)=
\begin{cases}
-1 - B(s^0 -\theta)^{\beta},& \theta\in [-h, s^0],\\[4pt]
-1 + A(\theta -s^0)^{\alpha}, &\theta\in (s^0, -\delta],\\[4pt]
\chi(\theta), & \theta \in (-\delta, 0]
\end{cases}
\quad =
\begin{cases}
-1 - B(-1 - \theta)^{\beta},& \theta\in [-h, -1],\\[4pt]
-1 + A(\theta+1)^{\alpha}, &\theta\in (-1, -\delta],  \\
\chi(\theta), & \theta \in (-\delta, 0]
\end{cases}
\end{equation}
with \(A,B>0\) and $\alpha,\beta\in (0,1)$, $\chi$ - any continuous connector,  $\chi(0)=1$ (to have $\varphi\in C[-h,0]$ and $\varphi(0)=1$), $\delta \in (0,1)$ fixed. For example, $\chi$ can be a linear function. 

One easily checks that 
\begin{equation}\label{red-solution}
x^r(t)=t+1, \quad t>0
\end{equation}
is a solution. 

  \begin{wrapfigure}{r}{0.5\linewidth} 
    \centering
    \includegraphics[width=0.80\linewidth]{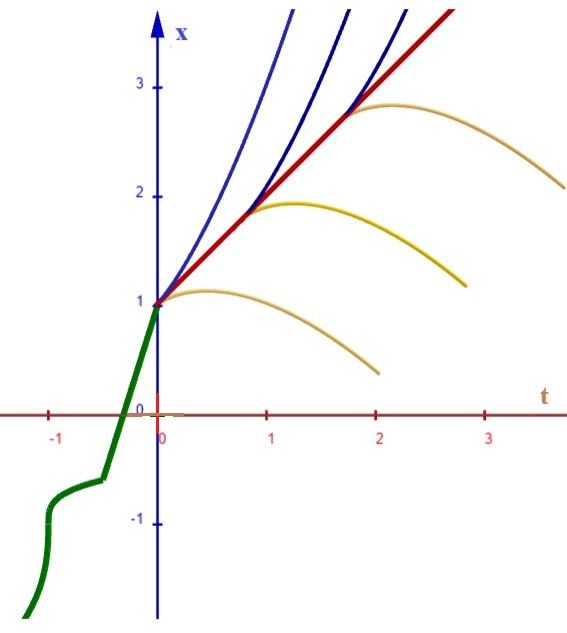}
    \caption{Example 3}
    \label{fig:3}
\end{wrapfigure}

One may verify that we have two families of solutions ($\tau\ge 0$ is a parameter).

The right-side family
\begin{equation}\label{yellow-tau-solutions}
x^\tau(t)=
\begin{cases}
1+t, & t\in [0, \tau],\\[4pt]
1 + t - C (t-\tau)^{\frac{1}{1-\alpha}},& t\in (\tau,\tau+\sigma],
\end{cases}
\end{equation}
with $C\equiv A^\frac{1}{1-\alpha}(1-\alpha)^\frac{1}{1-\alpha}$ and the left-side family 
  
\begin{equation}\label{blue-tau-solutions}
y^\tau(t)=
\begin{cases}
1+t, & t\in [0, \tau],\\[4pt]
1 + t + D (t-\tau)^{\frac{1}{1-\beta}},& t\in (\tau,\tau+\sigma],
\end{cases}
\end{equation}
with $D\equiv B^\frac{1}{1-\beta}(1-\beta)^\frac{1}{1-\beta}$. 
Here $\sigma>0$ is fixed.

\medskip

The above example gives an answer to question  (A): it may be infinite (even uncountable) number of continuous solutions to differential equation with 
{\it one} discrete state-dependent delay and a continuous initial function with only {\it one} point of non-Lipschitz. 

\medskip

\subsection{Colors of solutions} 

This short subsection is not about math results, but about intuition following each particular  solution/type of solution. 

The main idea: Let us color solutions 
depending on the behaviour of the  delayed argument $s(t)$ for  $t\in [t^1,t^2]$ (see (\ref{eq:main3-s})). 
We call $x(t)$ an {\it yellow solution} if its delayed argument $s(t)$ is strictly increasing for $t\in [t^1,t^2]$. 
We call $x(t)$ a {\it blue solution} if its delayed argument $s(t)$ is strictly decreasing for $t\in [t^1,t^2]$. 
We call $x(t)$ a {\it red solution} if its delayed argument $s(t)$ is constant for $t\in [t^1,t^2]$. 

One can easily see the colors of solutions constructed above. 
Clearly, $x^r(t)=t+1, \quad t>0$ (see (\ref{red-solution})) is a red solution. 
For example, solutions $x^\tau(t)$, given by (\ref{yellow-tau-solutions}), are red-yellow. More precisely, $x^\tau(t)$ are red for $t\in [0,\tau]$ and yellow for $t\in (\tau,\tau+\sigma]$.
Similarly,  solutions $y^\tau(t)$, given by (\ref{blue-tau-solutions}), are red-blue. More precisely, $y^\tau(t)$ are red for $t\in [0,\tau]$ and blue for $t\in (\tau,\tau+\sigma]$. 
For a particular case $\tau=0$, the corresponding $x^0(t)$ is yellow and $y^0(t)$ is blue. 

\begin{remark}\label{colors-const-delay}
The classical theory of differential equations with a constant delay $g(\cdot)=r>0$ can be colored as a {\it theory of  yellow solutions} because the corresponding delay arguments $s(t)=t-r$ is strictly increasing. Moreover, the theory of ordinary differential equations (without delay) is also "yellow" because $s(t)=t$ is also strictly increasing.

However, this does {\bf not} in any way divert close attention from other classes of solutions.
\end{remark}

\subsection{Red solutions}

As we saw in the previous examples 1-3, the huge number of solutions are connected to the existence of the "{\it red solution}". It attracts an additional attention to this kind of solutions. Let us consider some necessary conditions for the existence of red solutions. 

\smallskip

{\bf Lemma RS.} {\it 
Assume Cauchy problem (\ref{AR-sdd-2026_eq}), (\ref{AR-sdd-2026_ic}) has a continuous solution $\widetilde x(t), t\in [-h, \epsilon],$ $ \epsilon>0$ with the delayed argument $s(t) \equiv s^0 \equiv $ constant for $t\in [0,\epsilon]$ (see (\ref{eq:main3-s})). 
Then 

a) $f(\varphi(s^0))\neq 0.$

b) The state-dependent delay $g(x)$ is a {\tt linear} functional in a neighbourhood of $\varphi(0)=\widetilde x(0)$. 
The delay has a form $g(S)  = (S-\varphi (0))\cdot [f(\varphi(s^0))]^{-1}- s^0$ (in the corresponding neighbourhood). 

c) The red solution itself is given by 
\begin{equation}\label{AR-sdd-red-solution}
 \widetilde x(t) =  \varphi (0) + t\cdot f(\varphi(s^0)), \quad t\in [0, \epsilon]. 
\end{equation}
}

%
%
%

\begin{remark}
 The results of Lamma RS make all the systems with locally {\tt nonlinear} state-dependent delays $g(x)$  free of "{\it red solutions}". 
This property seems particularly useful  for models describing real world applications. 
\end{remark}

\noindent {\bf Corollaries RS.} \\
1) Combining items $a)$ and $c)$, one sees that a red solution is a {\it linear} function, but 
not a constant one. 
\\
2) 
No multiple red solutions may start at $t=0$. It follows from (\ref{AR-sdd-red-solution}), since all such a solution should use the same values $\varphi (0)$ and $f(\varphi(s^0))$ for the profile (\ref{AR-sdd-red-solution}). 
\\

\begin{remark}
 The previous Corollary RS 
gives even more interesting conclusion: the existence of a red solution is {\bf not} directly connected to the question of uniqueness/ nonuniqueness of solutions. More precisely, since any (one) red solution uses only values $\varphi (0)$ and $f(\varphi(s^0))$ for its shape (\ref{AR-sdd-red-solution}), then one can easily change initial function $\varphi$ for a new initial value problem with  $\widetilde \varphi$, but keeping the values $\widetilde \varphi (0)=\varphi (0)$ and $f(\widetilde  \varphi(s^0))=f(\varphi(s^0))$. As an example, we can consider equation (\ref{eq:main3}) and initial function 
\begin{equation}\label{lin-ic}
\widetilde \varphi(\theta)=2\theta+1, \quad \theta\in [-2,0].
\end{equation}
More detailed, let us compare (\ref{eq:main3}), (\ref{eq:main3-ic})
and (\ref{eq:main3}), (\ref{lin-ic}). We see that $\varphi (0)=\widetilde \varphi (0)=1$ and $f(\varphi(s^0))=f(\widetilde  \varphi(s^0))=1$, hence  solution (\ref{red-solution})
$x^r(t)=t+1, \quad t>0$ is the red solution of {\it both} initial value problems (\ref{eq:main3}), (\ref{eq:main3-ic})
and (\ref{eq:main3}), (\ref{lin-ic}). On the other hand, we saw that  (\ref{eq:main3}), (\ref{eq:main3-ic}) has mupliple solutions,  while (\ref{eq:main3}), (\ref{lin-ic}) has a unique solution. The uniqueness  follows from the Lipschitz property of $\widetilde \varphi(\theta)$ (\ref{lin-ic}). So, the mentioned 
 solution (\ref{red-solution}) of initial problem (\ref{eq:main3}), (\ref{lin-ic})  is  {\it unique} and {\it red}.  
\end{remark}

\smallskip

{\it Proof of Lemma RS} 
is surprisingly short.
If  $\widetilde x(t), t\in [-h, \epsilon], \epsilon>0$ with the delayed argument $s(t) \equiv s^0 \equiv $ constant for $t\in [0,\epsilon]$ is  a solution to  (\ref{AR-sdd-2026_eq}), (\ref{AR-sdd-2026_ic}), then the right hand side of (\ref{AR-sdd-2026_eq}) is constant and equals to $f(\varphi(s^0))$. Hence, by the left hand side of (\ref{AR-sdd-2026_eq}), one has a linear profile
\begin{equation}\label{sdd-linear}
 \widetilde x(t) = \widetilde x(0) + t\cdot f(\varphi(s^0)) = \varphi (0) + t\cdot f(\varphi(s^0)). 
\end{equation}
 
 {\it Case 1: $f(\varphi(s^0))=0$}. Hence $ \widetilde x(t) =  \varphi (0)$. In this case, by (\ref{eq:main3-s}), $s(t)$ cannot be constant. So, no such a solution may exist.

  {\it Case 2: $f(\varphi(s^0))\neq 0$}. 
  Let us introduce the new variable (see (\ref{sdd-linear})) $S \equiv \varphi (0) + t\cdot f(\varphi(s^0))$. In fact, it is the delayed argument of $g(x)$  for the particular solution $ \widetilde x(t)$. Hence $t=(S-\varphi (0))\cdot [f(\varphi(s^0))]^{-1}$ and $g(S)  = (S-\varphi (0))\cdot [f(\varphi(s^0))]^{-1}- s^0, $ which is linear.

%
%

\begin{remark}
Depending on the sign of $f(\varphi(s^0))\neq 0$, the linearity of the $g(x)$ is considered on a half-neighborhood of $\widetilde x(0)=\varphi(0)$ only. More precisely, if $f(\varphi(s^0))> 0$, then $U_{+}(\varphi(0))$ is considered. If $f(\varphi(s^0))< 0$, then $U_{-}(\varphi(0))$ is of interest. It makes the class of SDD function $g$ even more easily selected. 
\end{remark}

%
%
\begin{remark}
The simillar situation holds for a more general system 
$$\dot x(t) = \widehat{f}\left( x^1(t-g^1(x(t))),..., x^k(t-g^k(x(t))) \right).
$$
Assume the above system has a red solution such that 
all $s^i(t) \equiv t- g^i(x(t))\equiv s^{i0} =$ const, $i=1,...,k$, then all $g^i$ are {\it linear}.
\end{remark}

Discussing red solutions for more general system 
(\ref{sdd-winst-E}),  
 we have the following lemma. 
 \begin{lemma} Let $F$ be continuous and locally Lipschitz in its first argument, $g$ and $\varphi$ be continuous. Then the Cauchy problem (\ref{sdd-winst-E}),  (\ref{AR-sdd-2026_ic})  may not have more than one red solution. 
 \end{lemma}
 
 {\it Proof.} Assume 
 there are two red solutions of (\ref{sdd-winst-E}),  (\ref{AR-sdd-2026_ic}), denote them $x^{r,1}(t)$ and $x^{r,2}(t)$. Consider their delay arguments $s^1(t)\equiv s^1(0)=-g(x^{r,1}(0))=-g(\varphi(0)) =-g(x^{r,2}(0)) = s^2(0) \equiv s^2(t).$
Hence, by (\ref{sdd-winst-E}), both solutions satisfy the same ODE 
$\dot x(t) = \widehat{F}(x(t))$, where 
 $\widehat{F}(x) \equiv {F}(x, \varphi(-g(\varphi(0))).$ Since $ \widehat{F}(x)$ is  locally Lipschitz, the Cauchy problem 
 $\dot x(t) = \widehat{F}(x(t)),\, x(0)= \varphi(0)$ has the unique solution $ \widehat{x}(t)$. So, $x^{r,1}(t)\equiv \widehat{x}(t) \equiv  x^{r,2}(t)$. The proof is complete. 
 
%
%

\subsection{Graphs}

This subsection helps to visialise the main idea presented below.  
Let us discuss the graphs of solutions (\ref{red-solution}), (\ref{yellow-tau-solutions}), (\ref{blue-tau-solutions}). 
The standard approach is to draw all of them on the $(t,x)$-plane.

However, we mentioned the importance of the behaviour of the delayed argument $s(t)\equiv t- g(x(t)), t\ge 0$ and even proposed to color solutions according to that behaviour. The last  was an attempt to compensate the lack of information about $s(t)$ on the $(t,x)$-plane. It looks natural to draw the graphs of solutions in the $(t,s,x)$ system of coordinates. It is not difficult to do and then rotate the 3D-curves to see the standard $(t,x)$ and $(t,s)$ behaviour and an additional 3D properties. An example of such a property we formulate in the following theorem. 

\begin{theorem}\label{AR-2026-sdd-graph}
The graphs of all solutions (\ref{AR-sdd-2026_eq}), (\ref{AR-sdd-2026_ic}) 
drawn in the $(t,s,x)$ coordinate system, belong to the surface $ -t +s +g(x)=0  $. In a particular case of a (locally) linear SDD $g(x)=ax+b$, solutions (locally) belongs to the {\it plane} $-t+s+ax+b=0$.

\end{theorem}

\begin{corollary}\label{AR-2026-sdd-graph-cor}
The graphs of all solutions (\ref{red-solution}), (\ref{yellow-tau-solutions}), (\ref{blue-tau-solutions}) of the Cauchy problem 
(\ref{eq:main3}), (\ref{eq:main3-ic}) 
(with arbitrary parameters $A,B>0; \alpha, \beta \in (0,1); \tau\ge 0$), drawn in the $(t,s,x)$ coordinate system, belong to the plane $-t+s+x=0$.  

\end{corollary}

{\it Proof.} Since the proof of theorem \ref{AR-2026-sdd-graph} and Corollary \ref{AR-2026-sdd-graph-cor} are essentially the same, we perform it on the key example (\ref{eq:main3}), (\ref{eq:main3-ic}) to  emphasize the main idea.  We need to  write explicitly solutions in the $(t,s,x)$-coordinates. 

In case $\tau>0$ and $t\in [0,\tau]$, all the solutions coincide  (the red part). So we start with the red solution (\ref{red-solution}) ($x^{r}(t)=t+1, t>0$)

\begin{equation}\label{red-solution-3D}
x^{r,3D}(t)= 
\left(
\begin{array}{l} t  \\ -1  \\ 1+t  
\end{array}
\right).
\end{equation}
A straightforward check gives $-t+(-1)+ (1+t)=0$. 

Now we should check the property for (\ref{yellow-tau-solutions}), (\ref{blue-tau-solutions}) and $t>\tau$. 
The right-side family (\ref{yellow-tau-solutions}) for $t> \tau$   looks as 
\begin{equation}\label{yellow-tau-solutions-3D}
x^{\tau,3D}(t)=
\left(
\begin{array}{l}
t\\
t-g\left( 1 + t - C (t-\tau)^{\frac{1}{1-\alpha}}\right)\\
1 + t - C (t-\tau)^{\frac{1}{1-\alpha}}
\end{array}
\right) 
=
\left(
\begin{array}{l}
t\\
 -1 +  C (t-\tau)^{\frac{1}{1-\alpha}}\\
1 + t - C (t-\tau)^{\frac{1}{1-\alpha}}
\end{array}
\right),
\end{equation}
with $C\equiv A^\frac{1}{1-\alpha}(1-\alpha)^\frac{1}{1-\alpha}$.  
So, $-t+(-1 +  C (t-\tau)^{\frac{1}{1-\alpha}})+(1 + t - C (t-\tau)^{\frac{1}{1-\alpha}}= 0.$

The left-side family (\ref{blue-tau-solutions})   for $t> \tau$   looks as 
  
\begin{equation}\label{blue-tau-solutions-3D}
y^{\tau,3D}(t)=
\left(
\begin{array}{l}
t\\
t-g\left( 1 + t + D (t-\tau)^{\frac{1}{1-\beta}}\right)\\
1 + t + D (t-\tau)^{\frac{1}{1-\beta}}
\end{array}
\right) 
= 
\left(
\begin{array}{l}
t\\
-1 - D (t-\tau)^{\frac{1}{1-\beta}}\\
1 + t + D (t-\tau)^{\frac{1}{1-\beta}}
\end{array}
\right)
\end{equation}
with $D\equiv B^\frac{1}{1-\beta}(1-\beta)^\frac{1}{1-\beta}$. 
So, $-t+(-1 - D (t-\tau)^{\frac{1}{1-\beta}})+(1 + t + D (t-\tau)^{\frac{1}{1-\beta}})=0 $. 
It completes the proof. 

%
%

\begin{remark}
This linear pattern  may looks unexpected in context of essentially nonlinear delay equation (\ref{eq:main3}), but we remind the locally linear form of SDD functional $g(x)=|x|$ in a neighbourhood of $\varphi(0)=x(0)=1$ and the properties stated in Lemma RS. 
An interesting property is that this is {\it the same} plane for all the  parameters  $A,B>0; \alpha, \beta \in (0,1); \tau\ge 0$. 
\end{remark}

\section{Strictly monotone delay arguments}

As we saw in the 
examples above (remind also Remark~\ref{colors-const-delay}), the monotonicity of $s(t)$ may play an important role. In this Section we discuss an interesting  result \cite[p.621, Theorem 3.1]{Winston-JMAA-1974} by E. Winston. It is formulated for a more general system (the right hand side of the system also contains $x(t)$), so we follow the original framework. 
Now we consider (\ref{sdd-winst-E}). 

{\bf Proposition 1.} 
\cite[p.621, Theorem 3.1]{Winston-JMAA-1974}. 
{\it 
Suppose $F: D \times D \to \mathbb{R}^n$ is locally Lipschitz continuous,
and $g: D \to \mathbb{R}^{+}$ has a locally Lipschitz continuous derivative. If there exists
$\eta > 0$ such that
\begin{equation}\label{sdd-winst-*}
|y| < \eta \Rightarrow \langle g^\prime (x), F(x,y) \rangle < 1, 
\end{equation} 
then solutions of (\ref{sdd-winst-E}) with continuous initial functions $\varphi$ such that $||\varphi||_C<\eta$, are
unique.
}

\medskip

Following the line of the proof of this theorem, one sees that condition (\ref{sdd-winst-*}) 
is used to show the (strict) monotonicity of the delay argument $s(t)$. In the original proof \cite[p.621, Theorem 3.1]{Winston-JMAA-1974} it is a {\it strictly  increasing} function (for yellow solutions). We can extend the result to cover both cases yellow and blue solutions.

\smallskip

{\bf Proposition 2.} 
{\it 
Suppose $F: D \times D \to \mathbb{R}^n$ is locally Lipschitz continuous,
and $g: D \to \mathbb{R}^{+}$ has a locally Lipschitz continuous derivative. 
If 
\begin{equation}\label{sdd-AR-prop2}
\langle g^\prime (\varphi(0)), F(\varphi(0),\varphi(-g(\varphi(0)))) \rangle 
 \neq 1,  
\end{equation} 
then solutions of (\ref{sdd-winst-E}) with continuous initial functions $\varphi$  are
unique. Here $\langle \cdot, \cdot \rangle$ is the inner product in $\mathbb{R}^n$. 
}

\begin{remark} We can reformulate Proposition 2 in a way similar  to Proposition 1. To do it we just extend  (\ref{sdd-winst-*}) to the following condition: there exists
$\eta > 0$ such that
\begin{equation}\label{sdd-winst-*-ext}
|y| < \eta \Rightarrow \langle g^\prime (x), F(x,y) \rangle \neq  1. 
\end{equation} 
From our point of view, the property of smallness $|y| < \eta$, which is used in (\ref{sdd-winst-*}) and (\ref{sdd-winst-*-ext}) is not essential, while  (\ref{sdd-AR-prop2}) gives the essential condition for the (strict) monotonicity of the delay argument $s(t)$.
\end{remark}

\smallskip

{\it Proof of Proposition 2.} 
The existence of solution(s) to the problem (\ref{sdd-winst-E}), (\ref{AR-sdd-2026_ic}) follows from the Brouwer's fixed-point theorem. 
From that classical proof one knows that there are $\widetilde \alpha, \widetilde \beta >0$ such that solution(s) belong to the set 
$A\equiv \{ x(t)\in C([-h,\widetilde \alpha];\mathbb{R}^n) : x_0=\varphi,\, \max_{t\in [0,\widetilde \alpha]} ||x(t)-\varphi(0)||\le \widetilde \beta \}. $ 
For any $x(\cdot)\in A$ let us consider, see (\ref{sdd-dot-s}) below, the function 
$$\Phi(t,x(\cdot))\equiv 1 - \langle g^\prime (x(t)), F(x(t),\varphi(t-g(x(t)))) \rangle.
$$ 
It is continuous in $t\in [0,\widetilde \alpha]$ and, by (\ref{sdd-AR-prop2}), $\Phi(0,x(\cdot))\neq 0.$ By the classical theorem, 
there exist $\widehat \alpha \in (0,\widetilde \alpha], \gamma>0$ such that $|\Phi(t,x(t))|\ge \gamma>0$ 
for all $t\in [0,\widehat \alpha]$. 
Since (along a solution) 
\begin{equation}\label{sdd-dot-s}
\dot s(t) = 1 - \langle g^\prime (x(t)), \dot x(t) \rangle 
 = 1 - \langle g^\prime (x(t)), F(x(t),\varphi(s(t))) \rangle 
\end{equation} 
and $\dot s(t)= \Phi(t,x(t))$, the property $|\dot s(t)|\ge \gamma>0$ gives a strict monotonicity of $s(t)$. 
Hence there exists the inverse $t=t(s)$ and it is  differentiable with $\frac{dt}{ds} 
= \left[ \frac{ds}{dt} \right]^{-1}
$ (bounded). 

In this case, one defines $w(s)\equiv x(t)$, which is $w(s)\equiv x(t) = x(t(s)).$ At this moment the visialization in $(t,s,x)$ plane is particularly useful. 
Using $\dot s(t)\neq 0$ (in both cases  $\dot s(t)> 0$ and  $\dot s(t)< 0$), one has
$$\dot w(s) = \frac{dw}{ds}(s)= \frac{dx}{dt}\cdot \frac{dt}{ds} 
=  \frac{dx}{dt} \cdot \left[ \frac{ds}{dt} \right]^{-1}
$$
$$ = F(x(t), \varphi(s(t))) \cdot \left[ 1 - \langle g^\prime (x(t)), F(x(t),\varphi(s(t))) \rangle \right]^{-1}
$$
$$ = F(w, \varphi(s)) \cdot \left[ 1- \langle g^\prime (w), F(w,\varphi(s)) \rangle \right]^{-1} \equiv G(s,w). 
$$
Here $G$ is continuous in $s$ and locally Lipschitz in $w.$ 

Hence, the Cauchy problem for  
the ODE $ \frac{dw}{ds}(s)= G(s,w)$ with initial conditions $w(s^0)=x(0)$ has a {\it unique}
solution $w(s)\equiv x(t)$. As before, $s^0\equiv -g(x(0))<0$. 
This solution is the solution to (\ref{sdd-winst-E}), (\ref{AR-sdd-2026_ic}). 
The proof  of Proposition 2 is complete.   

%

\begin{remark}
The Cauchy problem for the ODE is solved in both directions (increasing $s$ if  $\dot s(t)> 0$ and decreasing $s$ if $\dot s(t)< 0$).
\end{remark}

\section{Concluding remarks}

1) We distinguish an immediate non-uniqueness and waiting-time non-uniqueness. Solutions (\ref{red-solution}) and (\ref{yellow-tau-solutions}), (\ref{blue-tau-solutions}) with $\tau=0$ represent immediate non-uniqueness, while (\ref{yellow-tau-solutions}), (\ref{blue-tau-solutions}) with $\tau>0$ represent  waiting-time non-uniqueness. 

\medskip

\noindent 2) Remind the question (B).  Which properties of an equation and an initial function are "responsible" for the non-uniqueness of continuous solutions? 
It looks that \smallskip

a) the behaviour of delay argument $s(t)$ is important. 
 \smallskip
 
b) Different non-Lipschitz properties of $\varphi$ on both half-neighbourhoods  of $s^0$ (point of non-
Lipschitz) may give birth to different families of solutions. See example 3.
 \smallskip
  
c) Waiting-time non-uniqueness may occur only in cases when either the red solution exists or the delay argument $s(t)$ leaves and latter returns to the point $s^0$.  

%
%

\bigskip

\bigskip

\hrule \hfill

\bigskip

\hfill Submitted

\bigskip

\hfill March 30, 2026

\bigskip

\hfill Kharkiv, Ukraine
%
%

%

\end{document}